\documentclass[12pt]{article}
\usepackage{graphicx}
\usepackage{amsthm}
\usepackage{amsmath}
\usepackage{amsfonts}
\newtheorem*{theorem}{Theorem}
\newtheorem{lemma}{Lemma}
\newtheorem*{coroll}{Corollary}
\newcommand{\R}{\mathbb{R}}
\newcommand{\Z}{\mathbb{Z}}

\author{S.\,Duzhin, M.\,Shkolnikov}
\title{Bipartite knots}
\date{}

\begin{document}
\maketitle

\begin{abstract}
We give a solution to a part of Problem 1.60 in Kirby's list of open problems 
in topology \cite{Kir} thus answering in the positive the question raised in 
1987 by J.\,Przytycki \cite{PP}.
\end{abstract}

\section{Problem}
We will call \textit{bipartite} a knot that can be represented by a
\textit{matched} diagram, that is, a diagram whose crossings are split in
pairs of the types depicted in Fig.~\ref{fig:p1}.
The pairs in the upper line are said to be \textit{positive}, those in the
lower line, \textit{negative}. Note the signs of the pairs do not change
when the orientation on the knot is reversed. Note, moreover, that if the
crossings of an unoriented knot are split into matched unoriented pairs,
then, introducing any orientation, we always get counter-directed pairs
shown in Fig.~\ref{fig:p1}.

\begin{figure}[htbp]
	\centering
		\includegraphics[width=0.5\textwidth]{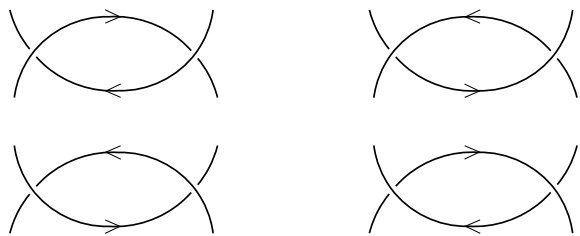}
	\caption{Matched pairs}
	\label{fig:p1}
\end{figure}

\textbf{Examples.}

1. Any rational knot has a matched diagram, because any rational number can be
represented by a continued fraction with even (positive or negative)
denominators (see \cite{DS}).

2. The standard diagram of the knot $8_{15}$ (which is not rational)
\begin{center}
    \includegraphics[height=30mm]{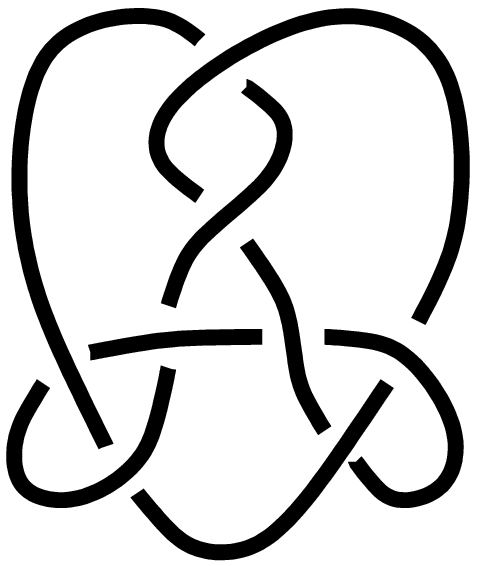}
\end{center}
can be easily transformed to a matched form:
\begin{center}
    \includegraphics[height=30mm]{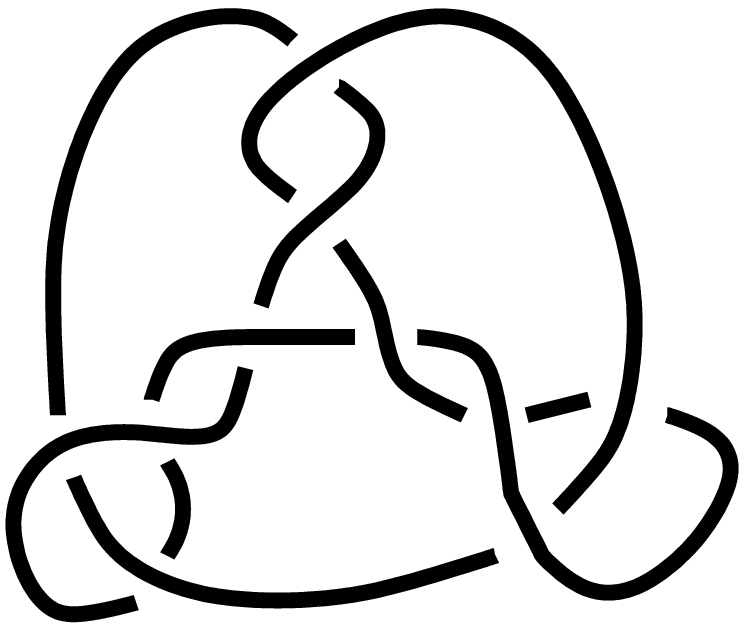}
\end{center}

3. We managed to find matched diagrams for all table knots with up to 8
crossings, save for the knot $8_{18}$.

The problem raised by Jozef Przytycki in 1987 is to investigate which knots
possess a matched diagram.  This question appears in the well-known
collection ``Open problems in
topology'' maintained by Rob Kirby \cite{Kir}, as part of Problem 1.60.
More exactly, Conjecture 1(a) therein (belonging to Przytycki \cite{PP}) 
reads: ``There are oriented knots without a matched diagram''.  As the
reader understands, the word ``oriented'' can be here omitted without any
loss of meaning.  This conjecture stayed open for 24 years, notwithstanding 
the effort of several
excellent mathematicians, including its author and J.~Conway \cite{APR}.  
We give a positive
solution to the conjecture, that is, demonstrate that some knots, e.g.  
pretzel knot with parameters $(3,3,-3)$, are not bipartite.  In the next 
section, we introduce our main construction, which also explains the meaning 
of the word ``bipartite'' in this context.

\section{Chord diagram of a bipartite knot}

Consider a matched diagram of a knot $K$.
Replace every matched pair of crossings by two parallel segments,
directed as the knot and joined by a common perpendicular, see 
Fig.~\ref{fig:p2}. 
\begin{figure}[htbp]
	\centering
		\includegraphics[width=0.5\textwidth]{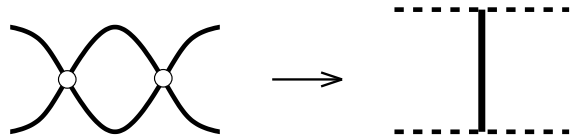}
	\caption{Local transformation of a matched diagram}
	\label{fig:p2}
\end{figure}

The parallel segments are then joined by the remaining
fragments of the knot diagram into a simple closed line on the plane, 
straightenable into a circle, whereas the common perpendiculars become 
\textit{chords}. An example for the matched diagram of the knot $8_{15}$
mentioned above, is given in Fig.~\ref{fig:md-cd}, where the inner and outer
chords are changed places: this leads to turning the knot diagram inside
outside with respect to some point and does not alter the isotopy type of
the knot.
\begin{figure}[htbp]
	\centering
	\includegraphics[height=4cm]{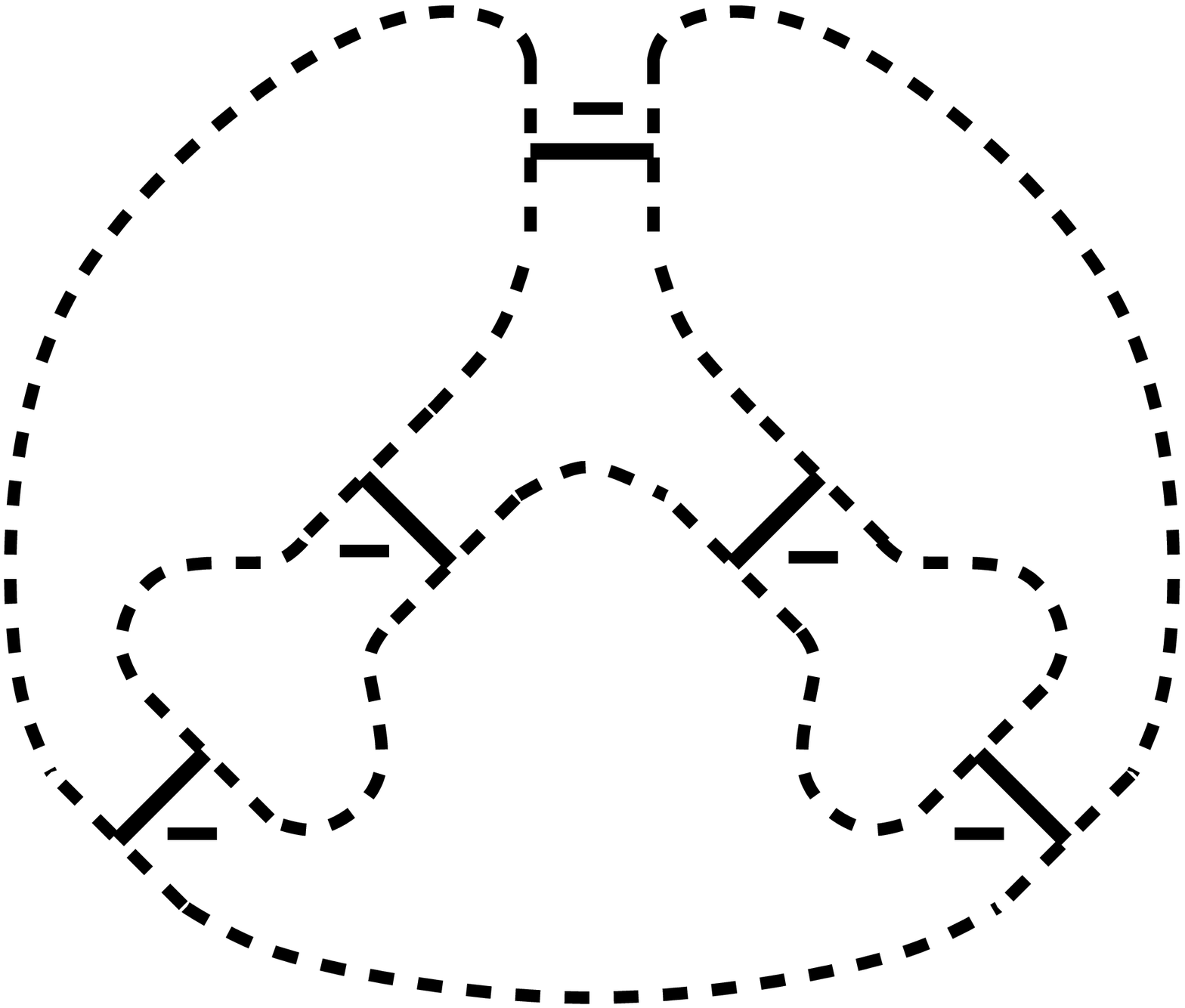}
	\quad
	\raisebox{19mm}{\includegraphics[width=1cm]{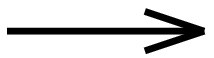}}
	\quad
	\raisebox{5mm}{\includegraphics[height=3cm]{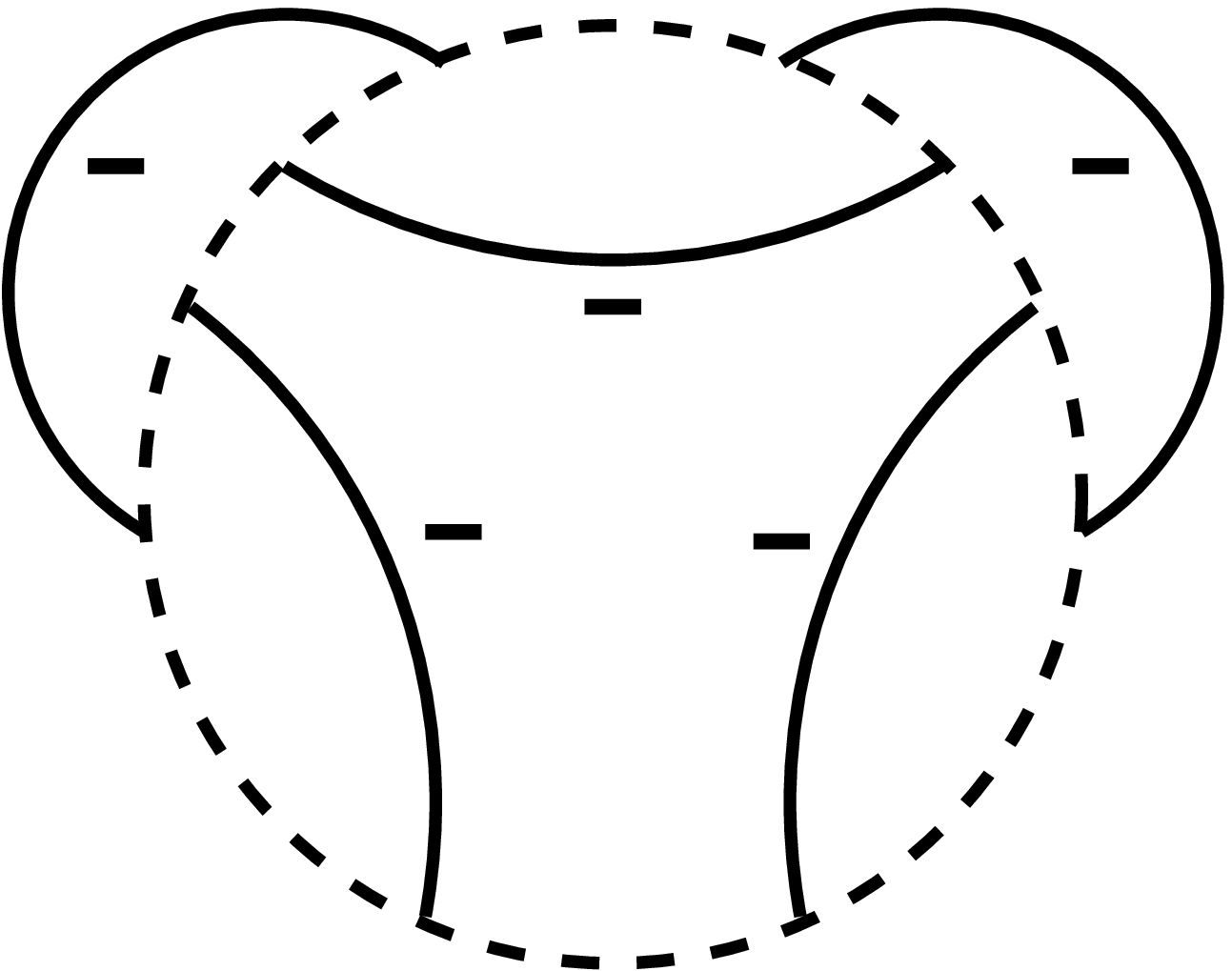}}
	\caption{Matched diagram $\to$ chord diagram}
	\label{fig:md-cd}
\end{figure}

The chord diagrams obtained in this way, are rather special:
the set of all chords is split into two parts (inner chords and outer
chords), so that the chords in each part do not intersect between
themselves, and the intersection graph of the whole diagram is bipartite.
 
This procedure is reversible: from a bipartite signed chord diagram
one can reconstruct the knot diagram in a unique way (see
Fig.~\ref{fig:cd-md}).
\begin{figure}[htbp]
	\centering
	\raisebox{5mm}{\includegraphics[height=3cm]{k8-15cd2.eps}}
	\quad
	\raisebox{19mm}{\includegraphics[width=1cm]{toto.eps}}
	\quad
	\includegraphics[height=4cm]{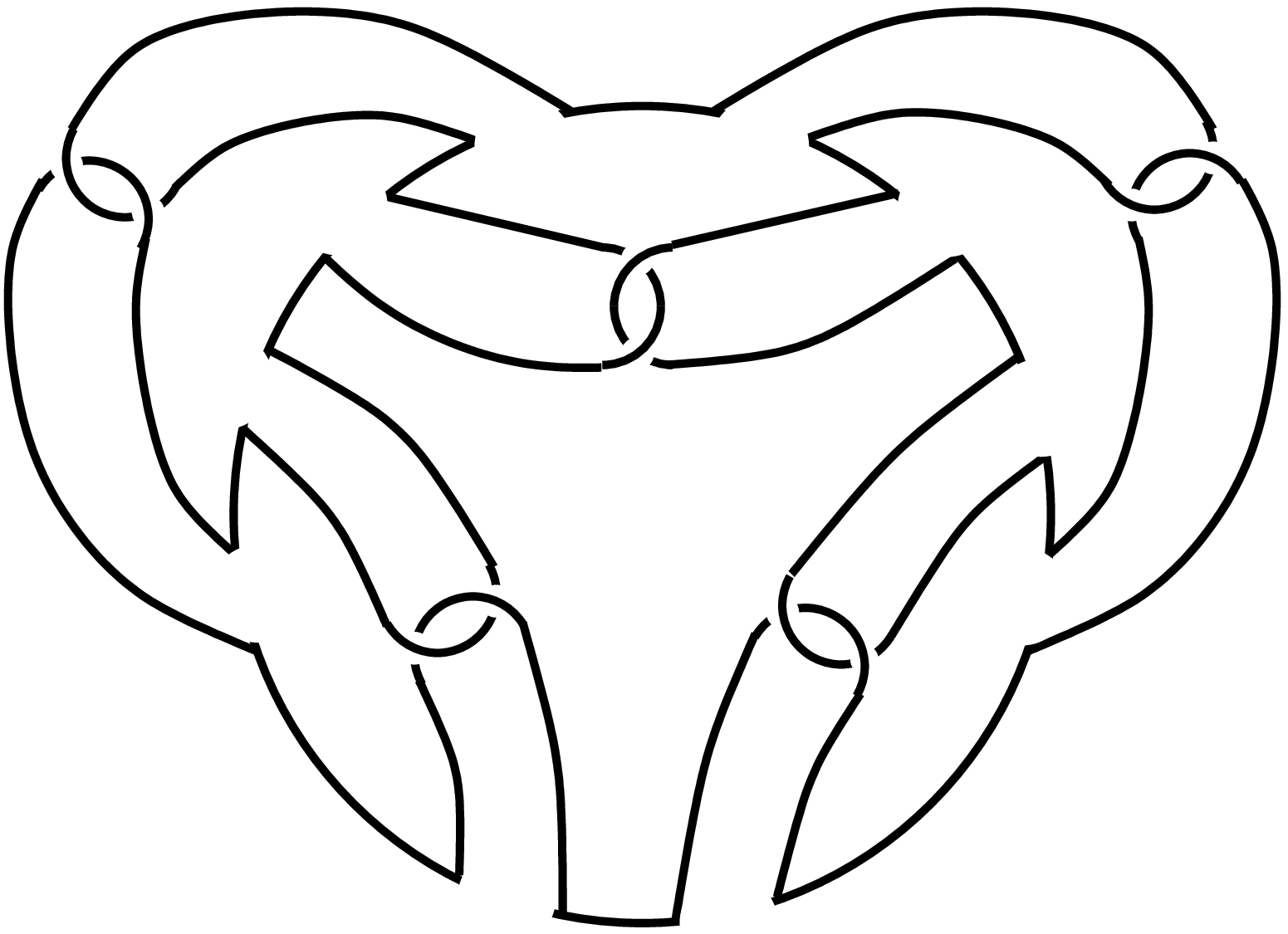}
	\caption{Chord diagram $\to$ matched diagram}
	\label{fig:cd-md}
\end{figure}

\section{Seifert surfaces}

A Seifert surface $S$ of a knot is a compact oriented surface embedded
in $\R^3$ so that its boundary is the given knot.
Choosing a basis in $H_1(S)$, one can construct a matrix of the bilinear
form $lk\circ (id,\alpha)$, where $lk$ is the linking number, and $\alpha$
is a small shift in the positive direction along the normal of $S$.
This matrix is called a Seifert matrix of the given knot. 

There is a standard procedure to construct a Seifert from any diagram, using
checkerboard coloring and Seifert circles. For matched diagrams there exists
a different construction, which is crucial for our needs:
it yields a Seifert matrix of a special type, which, in turn, produces
an Alexander matrix with extraordinary properties.

\begin{lemma}\label{lemma1}
Any bipartite knot has a Seifert surface such that its Seifert matrix
has the form
$\begin{pmatrix} E & 0 \\ I & F \end{pmatrix}$, where $I$, $0$, $E$, $F$ are
square matrices of the same size, $I$ is a unit matrix, $0$ is a zero matrix, 
and $E$ and $F$ are both symmetric integer matrices.
\end{lemma}
\begin{proof}
Consider a bipartite knot $K$ and its plane diagram,
construct the chord diagram as indicated above. Start constructing the
Seifert surface from the inner circle of the chord diagram, out of which we
cut out every chord together with a small open neighborhood 
and glue instead a double twisted band,
so that the direction of the twists corresponds to the sign of the chord 
(see Fig.~\ref{fig:p3}). 
\begin{figure}[htbp]
	\centering
	\includegraphics[width=0.5\textwidth]{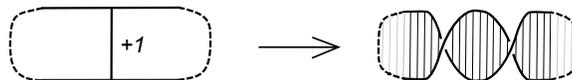}
	\caption{Construction of Seifert surface: inner chords}
	\label{fig:p3}
\end{figure}

So far the surface remains orientable, and its boundary follows the knot as
much as it can. Now we must add the bands along the outer chords. Here one
must be cautious, because simply connecting the ends of the two half-chords
by two half-twisted bands results in an unorientable surface. We will do as 
follows: first attach
a band along each outer chord, then, around the middle of that band, we
attach a perpendicular small band, which is twice twisted opposite to
the sign of the chord.  
\begin{figure}[htbp]
	\centering
	\includegraphics[width=0.7\textwidth]{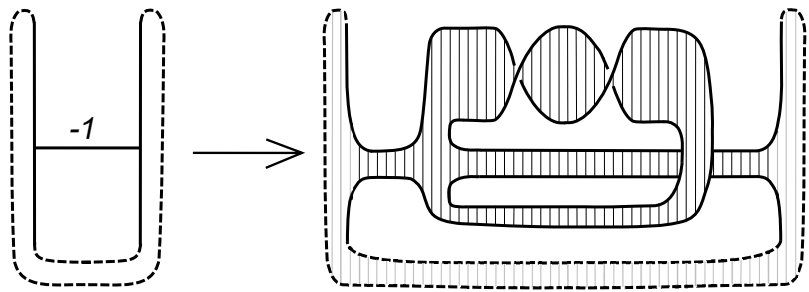}
	\caption{Construction of Seifert surface: outer chords}
	\label{fig:p4}
\end{figure}

We show this procedure on Fig.~\ref{fig:p4} for one chord on big scale and
on Fig.~\ref{seifert} for a whole Seifert surface of a certain knot.
In the latter picture, the narrow twice twisted bands on the left and on the
right should be though of as lying above the surface of the corresponding
perpendicular wide bands. The thick solid lines indicate the boundary of the
Seifert surface; the four small dashed segments show parts of the visible
contour of the surface which do not belong to its boundary. The two sides of
the Seifert surface (which is two-sided by definition) are indicated by
different shades of gray.
\begin{figure}[htbp]
	\centering
	\includegraphics[width=6cm]{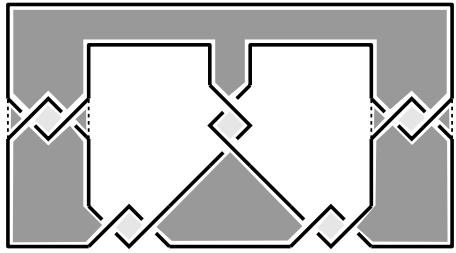}
	\caption{Seifert surface for a matched diagram of knot $8_{15}$}
	\label{seifert}
\end{figure}

Let $n$ be the number of outer chords (if this number is greater than the
number of inner chords, we can turn the chord diagram inside outside to
simplify computations). Then, as a basis of $H_1(S)$, one can take the set
$e_1,\dots,e_n,f_1,\dots,f_n$ corresponding to the outer chords and shown in
Fig.~\ref{fig:p5}.

\begin{figure}[ht]
	\centering
	\includegraphics[width=0.7\textwidth]{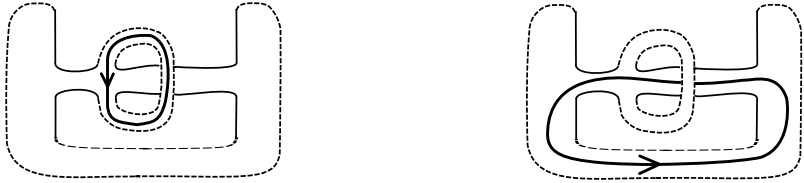}
	\caption{Cycles $e_k$ (left) and $f_k$ (right)}
	\label{fig:p5}
\end{figure}

It follows that the Seifert matrix for this Seifert surface (see \cite{Lik})
has the form
$$
  M=\begin{pmatrix} E & 0 \\ I & F \end{pmatrix},
$$
where $0$, $I$, $E$, $F$ are matrices of size $n \times n$, $I$ is the unit
matrix, $0$ is the zero matrix, and
$F_{i,j}=lk(f_i,f_j^+)$, $E_{i,j}=lk(e_i,e_j^+)$. 
It is clear that $E$ is a diagonal matrix with numbers $\pm 1$ on the
diagonal (the sign is inverse to the sign of the outer chord number $k$). 
The cycles $f_i$ and $f_j$ can be chosen so that not to have common points,
if $i\not=j$, therefore, 
$lk(f_i,f_j^+)=lk(f_i,f_j)=lk(f_j,f_i)=lk(f_j,f_i^+)$ and thus the matrix
$F$ is symmetric.

This construction shows that, on a practical side, it is advisable to turn
the diagram inside outside, if the number of outer chords is bigger than
that of inner chords --- as we did before for the example knot $8_{15}$. 
\end{proof}

\section{Alexander matrices}

The determinant of the Alexander matrix $A=tM-M^\top$ is equal to the
Alexander polynomial of the knot $K$; it is an element of the ring of
Laurent polynomials $\Z[t,t^{-1}]$, determined up to multiplication
by invertible elements of the ring, that is, monomials $\pm t^m$.
The Alexander matrix is not determined by the knot uniquely; in fact, 
to any knot there corresponds a big family of Alexander matrices
related between themselves by a set of equivalence transformations 
which is well known (see \cite{Lik}. In particular, even the size of the
matrix $A$ is not invariant. What is invariant, however, is the
sequence of Alexander ideals, the $m$-th ideal being defined as the ideal in
$\Z[t,t^{-1}]$ generated by all minors of an arbitrary Alexander matrix
of size $n-m+1$, where $n$ is the smallest among the number of columns
and rows in $A$ (see \cite{Lik}).

It is well known that the Alexander polynomial can be rewritten in terms of
the Conway variable $z^2=t+t^{-1}-2$; in general, this is no longer true about
the generators of all Alexander ideals.
In the case of bipartite knots, we can prove a stronger assertion.

\begin{lemma}\label{lemma2}
If the knot $K$ is bipartite, then there exists a square 
integer matrix $B$ such that the matrix $I+z^2B$ is an Alexander matrix for
$K$ (here $I$ is the unit matrix).
\end{lemma}
\begin{proof}
Consider the Seifert matrix $M$ from Lemma~\ref{lemma1}.
Put $A=tM-M^\top$, multiply the left block column by $-1$, the second
by $t^{-1}$, then interchange both columns. Using the symmetry of $E$ and
$F$, we get:
$$
tA-A^\top\leftrightarrow 
\begin{pmatrix} (t-1)E & -I \\ tI & (t-1)F \end{pmatrix}
\leftrightarrow \begin{pmatrix} I & (1-t^{-1})E  \\  
(1-t)F & I \end{pmatrix} .
$$ 
By a sequence of elementary transformations, we can make zero the upper
right block of this matrix, using its right lower block:
In doing so, we will be always adding polynomials 
$(1-t)a(1-t^{-1})b=-z^2ab$ to the elements of left upper block. 
In the end, the matrix will become
$$\begin{pmatrix} I+z^2B & 0  \\  (1-t)F & I \end{pmatrix} \leftrightarrow 
\begin{pmatrix} I+z^2B &0  \\  0 & I \end{pmatrix} \leftrightarrow I+z^2B .$$ 
\end{proof} 

To achieve our goal, it only suffices to prove one technical proposition.

\begin{lemma}\label{lemma3}
Let $p_1(x),\dots,p_n(x)\in\Z[x]$ be a set of ordinary polynomials
and $I=\langle p_1(z^2),\dots,p_n(z^2)\rangle$ be the corresponding 
ideal in $\Z[t,t^{-1}]$. Suppose that $I$ contains the binomial $(1+t)$. 
The the ideal $I$ is trivial: $I=\Z[t,t^{-1}]$.
\end{lemma}
\begin{proof}
It is clear that $(t+1)(1+t^{-1})=z^2+4\in I$. 
Then division gives $p_k(z^2)=p_k^0(z^2)(z^2+4)+a_k$, where $a_k\in\Z$ 
are some integers. Then our ideal coincides with $I=\langle z^2+4, a\rangle$, 
where $a=(a_1, \dots, a_n)$ is the greatest common divisor of all $a_i$'s. 
Expand the element $1+t$ in the new generators: 
$1+t=t^{-1}(t+1)^2q_1(t)+aq_2(t)$. Then $aq_2(t)$ is divisible by $1+t$. 
Therefore, 
$$
  1=(t+1)t^{-1}q_1(t)+a\frac{q_2(t)}{t+1} \in I,
$$
and the ideal is trivial.
\end{proof}

\section{Main result}

The last lemmas show that the Alexander ideals of bipartite knots cannot be
arbitrary. In particular, they are always generated by polynomials in $z^2$.

\begin{theorem}Let $K$ by a bipartite knot. 
If the Alexander ideal $I_m(K)$ is non-trivial, 
then it cannot contain the polynomial $1+t$.
\end{theorem} 
\begin{proof} 
This is a direct consequence of Lemmas \ref{lemma2} and \ref{lemma3}.
\end{proof}

This condition immediately give a series of knots which are not bipartite.

\begin{coroll}
Rolfsen table knots $9_{35}, 9_{37}, 9_{41}, 9_{46},9_{47},9_{48},9_{49}, 
10_{74}, 10_{75}, 10_{103}, 10_{155}, 10_{157}$ are not bipartite.
\end{coroll}

\begin{proof}
For the knot $9_{46}$, also known as the pretzel knot with parameters
$(3,3,-3)$, a detailed calculation of the second Alexander ideal is
available from \cite{Lik}. For other knots from the given list, we borrowed
the result from computer generated tables of Knot Atlas \cite{KA}.
\end{proof}

\section{From under the carpet}

Contrary to the universal tradition, we allow ourselves to raise the carpet
and explain how we actually arrived at this solution.

It was clear to us from the beginning that the rational knots are all
bipartite. Then we designed a procedure to very quickly compute the Conway
polynomial of a bipartite graph, starting from the corresponding signed
intersection graph. Looking through the table of all knots through 8
crossings, we managed to find the bipartite graphs that would give the same
Conway polynomials, and after one or two tries, using Knotscape \cite{Ksc}
and Knotinfo \cite{Kno}, we obtained a bipartite
representation for the corresponding knots. This worked for all knots, save
for $8_{18}$. Now, this is the only knot until 8 crossings with nontrivial
second Alexander ideal. We looked at other knots with nontrivial second
Alexander ideal and found some that cannot be expressed through the Conway 
variable $z^2$. 
On the other hand, we devised a procedure to represent the Alexander
matrix of a bipartite graph in terms of $z^2$.

After this work was finished, the second author (M.~Sh.) invented another
argument showing that a knot with second Alexander ideal 
$\langle 3,t^2+1 \rangle$,
e.g. the table knot $10_{122}$, cannot be bipartite.
A separate publication is being prepared to this end.

To summarize, we have presented a sufficient condition for a knot not
to have any matched diagram. We do not know, however, of any reasonable
necessary condition in terms of Alexander ideals.
The simplest knot which still stands our efforts, is $8_{18}$.

\end{document}